\documentclass[11pt]{amsart}
\usepackage{times}
\usepackage[T1]{fontenc}
\usepackage{amssymb, amsthm, amsmath}
\usepackage{mathrsfs}

\usepackage[active]{srcltx}

\usepackage{todonotes}

\usepackage{setspace}
\usepackage{geometry,soul}

\usepackage{hyperref}

\DeclareMathOperator{\acl}{acl}
 
\DeclareMathOperator{\gl}{GL} 
\DeclareMathOperator{\aut}{Aut} \DeclareMathOperator{\id}{Id}

\DeclareMathOperator{\eq}{eq}

\newtheorem{introtheorem}{Theorem}

\newtheorem{theorem}{Theorem}[section]

\newtheorem{claim}{Claim}[theorem]

\newtheorem{corollary}[theorem]{Corollary}

\newtheorem{fact}[theorem]{Fact}
\newtheorem{lemma}[theorem]{Lemma}

\newtheorem{proposition}[theorem]{Proposition}

\newtheorem*{gen-dif}{\fbox{{\large A}} \hypertarget{Agen-dif}{Gen-Dif}}

\newtheorem*{min-balln}{\fbox{{\large A}} \hypertarget{Amin-ball}{Cballs}}

\theoremstyle{definition}
\newtheorem{definition}[theorem]{Definition}
\newtheorem{example}[theorem]{Example}
\newtheorem{remark}[theorem]{Remark}

\newtheorem{notation}[theorem]{Notation}

\newcommand{\Nn}{{\mathbb{N}}}

\newcommand{\Zz}{{\mathbb {Z}}}

\newcommand{\m}{\textbf{m}}
\newcommand{\bk}{\textbf{k}}

\newcommand{\CK}{{\mathcal K}}
\newcommand{\CN}{{\mathcal N}}

\newcommand{\CM}{{\mathcal M}}

\newcommand{\CC}{{\mathcal C}}

\newcommand{\CO}{{\mathcal O}}

\newcommand{\CG}{{\mathcal G}}

\newcommand{\0}{\emptyset}

\newcommand{\rest}{\upharpoonright}

\renewcommand{\phi}{\varphi}

\newcommand{\ad}{\mathrm{Ad}}

\def\qp{\mathbb Q_p}

\def\dpr{\mathrm{dp\text{-}rk}}
\def\sub{\subseteq}

\newenvironment{claimproof}[1][\proofname]
  {%
    \proof[#1]%
  }
  {%
    \endproof%
  }

\title{Definably semisimple groups interpretable in $p$-adically closed fields}
\author{Yatir Halevi}
\address{Department of Mathematics\\ University of Haifa\\ 199 Abba Khoushy Avenue \\ Haifa \\Israel}
 \email{ybenarih@campus.haifa.ac.il}

\author{Assaf Hasson}
\address{Department of Mathematics, Ben Gurion University of the Negev, Be'er-Sheva 84105, Israel}
\email{hassonas@math.bgu.ac.il}

\author{Ya'acov Peterzil}
\address{Department of Mathematics, University of Haifa, Haifa, Israel}
\email{kobi@math.haifa.ac.il}

\date{\today}

\begin{document}

\thanks{The first author was partially supported by ISF grant No. 555/21 and 290/19. The second author was supported by ISF grant No. 555/21. The third author was supported by ISF grant No. 290/19.}

\begin{abstract}
    Let $K$ be a $p$-adically closed field and $G$ a group interpretable in $K$. We show that if $G$ is definably semisimple (i.e. $G$ has no definable infinite normal abelian subgroups) then there  exists a finite normal subgroup $H$ such that $G/H$ is definably isomorphic to a $K$-linear group. The result remains true in models of $\mathrm{Th}(\mathbb{Q}_p^{an})$.
 
\end{abstract}

\maketitle

\section{Introduction}

A $p$-adic semi-algebraic set is a finite boolean combination of sets of the form $\{x \in K^n: \exists y f(x)=y^m\}$ for $m\in \Nn$, $m\ge 2$ and $f\in \qp[x]$. By Macintyre's theorem, \cite{Macp-adic}, these are exactly the definable sets in the valued field $\mathbb Q_p$. Unlike the situation over the reals,  the class of $p$-adic semi-algerbaic sets is not closed under quotients by semi-algerbaic equivalence relations. For example, neither the value group, $\Gamma=K^\times /\CO^\times $, nor the set of closed $0$-balls, $K/\CO$, is semi-algebraic. A set obtained as such a quotient by a definable equivalence relation is called {\em interpretable}.

In \cite{HaHaPeVF} we showed, among others, that an  infinite field interpretable in a $p$-adically closed field $\CK$ is definably isomorphic to a finite extension of $K$. In particular, it follows that every infinite field interpretable in $\CK$ is, in fact, definable (and can be realised as a subfield of the ring $M_n(K)$ for a suitable $n$). In the present paper, we prove an analogous result for definably semisimple groups interpretable in $\CK$. Our main result is: 

\begin{introtheorem}\label{T: intro1}
Let $\CK$ be a  $p$-adically closed field  and let $G$ be an interpretable  definably semisimple group in $\CK$. Then there exists a finite normal subgroup  $H\trianglelefteq G$, defined over the same parameters as $G$, such that  $G/H$
is definably isomorphic to a definable $K$-linear group. 
\end{introtheorem}

In this paper, by a definably semisimple group, we mean an infinite group admitting no infinite definable normal abelian subgroups.  

We have been informed by J. Gismatullin,  I. Halupczok and D. Macpherson that in a recent unpublished work they characterise simple groups \emph{definable}  in certain henselian valued fields (including $p$-adically closed fields and their $1$-h-minimal analytic expansions).  Their work seems to combine with the present one to characterise definably simple groups interpretable in $p$-adically closed fields.

Observe that while the theory of $p$-adically closed fields has uniform finiteness for definable families of subsets of the valued field sort, the same is not true for the other sorts.  E.g., neither the value group, $\Gamma$, nor the sort of closed $0$-balls, $K/\CO$, has uniform finiteness. In this situation the notion of definable semisimplicity need not a-priori be elementary. It is one of the consequences of the present work, Corollary \ref{C: ss is fo}, that if $G$ is an interpretable definably semisimple group, in a model $\CK$ (namely the group $G(\CK)$ has no definable infinite normal abelian subgroups) then $G$ is definably semisimple in any elementary extension of $\CK$.

The proof of Theorem \ref{T: intro1} is a case analysis based on new invariants associated with interpretable groups, introduced in \cite{HaHaPeGps}, as well as some tools from Johnson's recent work, \cite{JohnTopQp}, on the topology associated with such groups. Elimination of imaginaries in $p$-adically closed fields was proved in \cite{HrMarRid}, however, as in both \cite{HaHaPeGps} and \cite{JohnTopQp}, the proof here avoids the general theory of elimination of imaginaries. 

Recall that in \cite[Lemma 7.10]{HaHaPeGps} we have shown that any set $X$ interpretable in $K$ is locally almost strongly internal to either $K$, $K/\CO$ or to $\Gamma$. I.e.,  there is an infinite definable $Y\sub X$ (strictly speaking, $Y$ is also interpretable) and a definable finite-to-one function from $Y$ into $D^n$ (some $n$) where $D$ is one of the above three sorts. We say that $X$ is {\em $D$-pure} if $D$ is the unique such sort, in every elementary extension of $\CK$.

The key ingredient (Theorem \ref{T:johnson question} below) in the proof of our main theorem is the following:
\begin{introtheorem} \label{T: main into}
    Let $G$ be an infinite group interpretable in a  $p$-adically closed field  $\CK$. 
    \begin{enumerate}
        \item Assume that  $G$ is $K$-pure and locally abelian with respect to Johnson's admissible topology (see Section \ref{ss: topologies}), namely there exists an open neighborhood of $e$ where all elements commute. Then there exists a definable abelian  normal $G_1\trianglelefteq G$ of finite index.
        \item If $G$ is not $K$-pure  then there exists a definable infinite normal abelian subgroup $N\trianglelefteq G$.
    \end{enumerate}
    In any case, if $G$ is locally abelian then it contains a definable infinite normal abelian subgroup.
\end{introtheorem}

 In fact, our results are more precise than stated, as we can give non-trivial lower bounds on the dp-rank of the normal abelian subgroups obtained in clauses (1) and (2). 

The main results of  this paper go through, essentially unaltered, to models of $\mathrm{Th}(\qp^{an})$ (or, indeed, to models of analytic expansions of finite extensions of $\qp$. See \cite{DenvdDr} for details). For ease of exposition, and in order to avoid technicalities, we have opted to work in $p$-adically closed fields. This is discussed in more detail in the remarks concluding the paper. 

\vspace{.2cm}

\noindent{\bf Previous work} 
We note recent work on interpretable groups in $p$-adically closed fields, by Johnson, \cite{JohnTopQp}, also together with  Yao,  \cite{JohnYao}, \cite{JohnYaoAbelian},  and with Guerrero, \cite{JohnGue}.  Further work is needed in order to understand the relation between our methods and the model theoretic tools studied there,  such as definable compactness, finitely satisfiable generics (fsg), definable $f$-generics (dfg), etc.

\vspace{.2cm}

\noindent\emph{Acknowledgement} We would like to thank J. Gismatullin, I. Halupczok and D. Macpherson  for sharing with us their unpublished work on simple groups definable in certain henselian fields. We also thank D. Macpherson for several conversations and useful suggestions, and E. Sayag for directing us to some useful references.  

\section{Preliminaries}

\subsection{Notation, conventions and basic definitions}\label{ss:notation}
Throughout, structures are denoted by  calligraphic capital letters, $\CM$, $\CN$, $\CK$ etc., and their respective universes by the corresponding Latin letters, $M$, $N$ and $K$. 

Tuples from a structure $\CM$ are always assumed to be finite, and are denoted by small Roman characters $a,b,c,\dots$. We apply the standard model theoretic abuse of notation writing $a\in M$ for $a\in M^{|a|}$. Variables will be denoted $x,y,z,\dots$ with the same conventions as above. We do not distinguish notationally between tuples and variables belonging to different sort, unless some ambiguity can arise. Capital Roman letters $A,B,C,\ldots$ usually denote small subsets of parameters from $\CM$. As is standard in model theory, we write $Ab$ as a shorthand for $A\cup \{b\}$. In the context of definable groups we will, whenever confusion can arise, distinguish between, e.g., $Agh:=A\cup\{g,h\}$ and $A\,g\!\cdot \! h:=A\cup \{g\!\cdot\! h\}$. 

By a partial type we mean a consistent collection of formulas. Two partial types, $\rho_1, \rho_2$ are equal, denoted $\rho_1=\rho_2$, if they are logically equivalent, i.e., if they have the same realizations in some sufficiently saturated elementary extension.

A \emph{$p$-adically closed field} is a field elementarily equivalent to  $\qp$ or to a finite extension of $\qp$  in the  language of valued fields, see, e.g., \cite{PreRoq} for more information.
\textbf{We let  $\CK$ denote a $(2^{\aleph_0})^+$-saturated  $p$-adically closed field}, and $K$ will always be its  valued field sort. Unless specifically written otherwise, we will always work in $\CK^{\eq}$.  
{\bf Henceforth, by ``definable'' we mean ``definable in $\CK^{\eq}$ using parameters'', unless specifically mentioned otherwise}. In particular, we shall not use ``interpretable'' anymore. A more detailed review of standard definitions and notation can be found in \cite[\S 2]{HaHaPeGps}.

We use freely properties of dp-rank (such as sub-additivity, invariance under finite-to-finite correspondences, invariance under automorphisms etc.), see the preliminaries sections of \cite{HaHaPeVF},\cite{HaHaPeGps} for a more detailed discussion. For an $A$-definable set $X$, an element $c\in X$  is \emph{$A$-generic} (in $X$), or \emph{generic over $A$},  if $\dpr(c/A)=\dpr(X)$.

For any valued field, $(K,v)$, we denote by $\CO_K$ (or just $\CO$ if the context is clear) its valuation ring. Its maximal ideal $\m_K$ (or $\m$) and $\bk_K:=\CO_K/\m_K$ (or just $\bk$) its residue field. The value group is $\Gamma_K$ (or just $\Gamma$). A closed ball in $K$ is a set of the form $B_{\geq \gamma}(a):=\{x\in K: v(x-a)\geq \gamma\}$ and likewise $B_{>\gamma}(a)$ for open balls. We will use the fact that $v$ extends naturally to $K/\CO\setminus \{0\}$ (by $v(a/\CO):=v(a)$ for any $a\notin \CO$), and use the same notation $B_{>\gamma}(x)$ and $B_{\ge \gamma}(x)$ for $x\in K/\CO$ in the obvious way. We will, however, reserve \textbf{the term ``ball'' in $K/\CO$ only to  such sets where $\gamma<\Zz$}.   A ball in $K^n$ (or in $(K/\CO)^n$) is an $n$-fold  product of $K$-balls (or $(K/\CO)$-balls) of {\bf equal radii}. 

As $\CK$ is a $p$-adically closed field, it is elementarily equivalent to some finite  extension,  $\mathbb{F}$, of $\mathbb{Q}_p$.  By saturation, we may assume that $(K,v)$ is an elementary extension of $(\mathbb{F},v)$.  Since its value group $\Gamma_{\mathbb{F}}$ is isomorphic to $\mathbb{Z}$, as ordered abelian groups,  we identify $\Gamma_{\mathbb{F}}$ with $\mathbb{Z}$ and view it as a prime (and minimal) model for $\Gamma$. We denote $\Zz_{Pres}$ the structure $(\Zz, +, <)$.

\begin{remark}
In \cite[\S3]{HaHaPeGps} we study the structure of $K/\CO$ in $p$-adically closed fields. In this context, it was helpful to work in a saturated model, expanding the language by constants for all  elements of (a copy of) $\mathbb{F}$.

Although the saturated model $\CK$ plays an important role in many of our proofs, the main theorems  of the present paper do not have any saturation assumptions. Thus,  a copy of $\mathbb F$ cannot be expected to exist in all our models (let alone be named). Whenever needed (especially in Section \ref{ss: lsi to K/O}), as part of the proof,  we bridge this gap in the assumptions. 
\end{remark}

\subsection{Some specialised terminology} 
By \cite{DoGoLi},  $p$-adically closed fields are dp-minimal, so every definable set in $\CK=\CK^{eq}$ has finite dp-rank. Using this basic fact, we remind some terminology from \cite{HaHaPeGps} that will be used throughout the paper: 

 The sorts $K, \Gamma$ and $K/\CO$ are referred to as the \emph{distinguished sorts}  (of $\CK$).  Given $\CK_0\equiv \CK$ a set $S$, definable in $\CK_0$, is \emph{locally almost strongly internal to a distinguished sort $D$} if {\bf in a sufficiently saturated elementary extension}
  there is a definable infinite set $X\sub S$ and a definable finite-to-one map $f: X\to D^n$, for some $n\in \Nn$. The set $X$ is called \emph{almost strongly internal to $D$}. If we can find a definable injection, $f:X\to D^n$, then $S$ is \emph{locally strongly internal to $D$} and $X$ is \emph{strongly internal to $D$}. By \cite[Lemma 7.10]{HaHaPeGps} every infinite set definable in $\CK$ is locally almost strongly internal to at least one of the distinguished sorts.

 A $D$-critical subsets of $S$ is a $K$-definable $D$-strongly internal  $X\sub S$ of maximal dp-rank.  The \emph{$D$-rank\footnote{In \cite{HaHaPeGps} this was called the $D$-critical rank of $S$.}} of $S$ is the dp-rank of any $D$-critical $X\sub S$.  Almost $D$-critical sets (and the corresponding almost $D$-rank) are defined by replacing ``$D$-strongly internal'' with ``almost $D$-strongly internal''. 
A definable set $S$ is \emph{$D$-pure} if it is locally almost strongly internal to $D$ but not to any other distinguished sort.   \\

 In several  of the results we need from  \cite{HaHaPeGps} it is assumed that the definable group $G$ is an \emph{(almost) $D$-group}, for some distinguished sort $D$.  The definition of $D$-groups, \cite[Definition 4.23]{HaHaPeGps}, is rather technical and not explicitly used anywhere in the present text, so we omit it. For the present text it suffices to know that, by \cite[Fact 4.25]{HaHaPeGps}, every definable group in $\CK$ which is locally almost strongly internal to $D$ is an almost $D$-group.  More importantly, by  \cite[Proposition 4.35]{HaHaPeGps} any infinite definable group $G$, locally almost strongly internal to $D$, contains a finite normal subgroup, $H$,  defined over the same parameters as $G$, such that  $G/H$ is a $D$-group. Every $D$-group is in particular locally strongly internal to $D$.

\subsection{The infinitesimal group $\nu_D$}

 One of the main results of \cite{HaHaPeGps} (Theorem 7.11) associates to any $D$-group a canonical type definable subgroup $\nu_D$ over $K$, strongly internal to $D$. The construction was done abstractly, but the current paper requires a more concrete description, that we give below. As the proof is not directly related to the main results of the present paper we differ it to  Appendix \ref{AS:nu}. Before doing so, we remind the following definitions. 

Recall our convention that a {\em ball} in $(K/\CO)^n$ is an $n$-fold product of valuative subballs of $K/\CO$ of (negatively) infinite and equal radii. Similarly, in $\Gamma$ we have:
\begin{definition}[\text{\cite[Definition 3.2]{OnVi}}]\label{D: boxes}
    A \emph{generalised box} around $a:=(a_1,\dots, a_n)\in \Gamma^n$ is the product of $n$ sets of the form $(b_i,c_i)\cap \{x_i: x_i-a_i\in P_{m_i}\}$ where both intervals $(b_i,a_i)$ and $(a_i,c_i)$ are infinite and $P_{m_i}$ is the predicate for $m_i$-divisibility. 
\end{definition}

Since $K/\CO$ and $\Gamma$ do not eliminate $\exists^\infty$ (see Section \ref{section 2.4}) the notions of balls in $(K/\CO)^n$ as well as that of  generalized boxes in $\Gamma^n$ are not definable in families. \\
 
 We can now state the lemma describing the construction of  $\nu_D$ in terms of balls (and generalised boxes): 
 
 \begin{lemma}\label{L: nu}
    Let $G$ be a $D$-group definable in $\CK$, for $D=K,\Gamma$, or $K/\CO$ and $n$ the $D$-rank of $G$. Then there exists a symmetric $D$-critical set, $X\sub G$,  with $\nu_D\vdash X$ and $f: X\to D^n$ a definable injection, such that: 
    \begin{enumerate}
        \item For $D=K$, $f(X)$ is a ball around $0$ and 
        \[
        \nu_K=\{f^{-1}(U): U\sub K^n \text{ a ball around $0$ }\}.
        \]
        \item For $D=\Gamma$,  $f(X)$ is a generalized box around $0$ and 
        \[
        \nu_\Gamma=\{f^{-1}(U): U\sub K^n \text{ a generalised box around $0$ }\}.
        \]
         Moreover, for every $x,y\in X$, if $xy^{\pm 1}\in X$ then $f(xy^{\pm 1})=f(x)\pm f(y)$.
         
        \item For $D=K/\CO$,   $X \le G$ is a subgroup,  $f$ is a definable injective group homomorphism, with $f(X)$ a ball around $0$, and 
        \[
        \nu_{K/\CO}=\{f^{-1}(U): U\sub K^n \text{ a ball around $0$ }\}.
        \]
    \end{enumerate}
 \end{lemma}

For the proof, see Appendix \ref{AS:nu}.  

\subsection{Dimension$^{eq}$}\label{section 2.4}
We recall that the valued field sort $K$ of $\mathcal K$ is a geometric structure. I.e., $\acl$ satisfies Steinitz exchange (e.g., \cite[Corollary 6.2]{HasMac}) and the quantifier  $\exists^\infty$ can be eliminated.

This last property, referred to sometimes as \emph{uniform finiteness}, means that given a definable family $\phi(x,y)$ and an $\aleph_0$-saturated model $\CM$ the set $\{b\in M: |\phi(M,b)|<\infty \}$ is definable. Gagelman, \cite{Gagelman}, extends  the $\acl$-dimension on $K$  to $\CK^{eq}$ as follows:  Given a definable equivalence relation $E$ on $K^n$ set 

\[\dim^{eq}(a_E/A)=\max\{\dim(b/A)-\dim[a]:b\in [a]\},\] where $\dim:=\dim_{\acl}$,
the $E$-equivalence class of $a$ is $[a]\sub K^n$ and   $a_E:=a/E\in K^n/E$.
For $Y\sub X/E$ defined over $A$, we define
\[\dim^{eq}(Y)=\max\{\dim^{eq}(a_E/A):a_E\in Y\}.\]

For a concise summary of the properties of $\dim^{eq}$ we refer to \cite[\S 2]{JohnTopQp}. In the present text we will mostly use additivity of $\dim^{eq}$: 
\[
\dim^{eq}(a,b/A)=\dim^{eq}(a/Ab)+\dim^{eq} (b/A).
\]
Note that $\dim^{eq}$ coincides with $\dim_{\acl}$ on definable subsets of $K^n$, and on tuples in $K$ over parameters from $\CK$.  For ease of notation, we use $\dim$ instead of $\dim^{eq}$ for imaginary elements as well.

Note also that for definable subsets of $K^n$ dimension is the same as dp-rank (\cite[Theorem 0.3]{Simdp}), a fact that we use without further mention. We also use the fact that, by sub-additivity of the dp-rank, it follows immediately from the definitions that $\dim(X)\leq \dpr(X)$ for any definable set $X$ in $\CK^{eq}$.  Note, also,  that $\dim(D)=0$ for $D=\Gamma$ and $D=K/\CO$.

Since dimension is preserved under definable finite-to-one functions, it follows that if $X$ is locally almost strongly internal to $K$ then $\dim(X)>0$. The converse follows from the fact that $K$ is the only distinguished sort of positive dimension. The equivalence can also be stated as follows.

\begin{lemma}\label{l: pure 0dim}
    A definable set $S$ is $K$-pure if and only if every definable $0$-dimensional  $X\sub S$ is finite. 
\end{lemma}
\begin{proof} 
    Assume that  $X\sub S$ is infinite and $0$-dimensional. By \cite[Lemma 7.19]{HaHaPeGps},  $X$ (and hence also $S$) is locally almost strongly internal to some distinguished sort $D$. Namely, there is a definable infinite $X_1\sub X$ and $f:X_1\to f(X_1)\sub D^n$ finite-to-one.  Since $\dim(X_1)\ge \dim (f(X_1))$,  necessarily $\dim (f(X_1))=0$ with $f(X_1)$ infinite. Hence,  $D\neq K$, so $S$ is not $K$-pure. 
    
    For the converse, assume that $S$ is not $K$-pure, witnessed by an infinite $X\sub S$ and $f:X\to D^n$ finite-to-one, for $D\neq K$.  Since $\dim(D)=0$ for  $D\neq K$,  it follows that  $\dim(f(X))=0$ and hence $\dim(X)=0$.   So, $X$ is infinite and $0$-dimensional.
\end{proof}

\subsection{Topologies}\label{ss: topologies}

By a definable topology on a definable set $X$, we mean a uniformly definable basis of open sets.   In \cite[Corollary 5.14]{HaHaPeGps}  we show that if $G$ is locally strongly internal to $K$ then $G$ can be endowed with a definable Hausdorff group topology, $\tau_K$.  In \cite{JohnTopQp}, Johnson equips, more generally,  any interpretable group $G$  with a (possibly discrete) definable Hausdorff group topology. Johnson's so-called \emph{admissible topology} is defined by the two following properties:
\begin{enumerate}
    \item  There is a definable manifold $Y$ and  a definable surjective continuous open map $Y\to G$. 
    \item  For every $a\in G$ there is a neighbourhood $U\ni a$ and a definable homeomorphism from $U$ onto an open subset of some $K^n$ (possibly, the one-point space $K^0$).
\end{enumerate}
See, \cite[Theorem 4.29]{JohnTopQp} for the details. In \cite[Theorem 5.10]{JohnTopQp} Johnson proves that any interpretable group $G$ admits a unique admissible group topology.   As can be readily seen from (2) above, if $\dim(G)=0$ then Johnson's \emph{admissible} topology is discrete. We will see below that if $\dim(G)>0$
then the $\tau_K$-topology coincides with the admissible topology.

Recall that the \emph{(almost) $K$-rank of $G$} is the maximal dp-rank of a definable $X\sub G$ (almost) strongly internal to $K$.  The equivalence of (2) and (3) below allows us to avoid some technical issues that may occur in groups which are  locally {\em almost} strongly internal to $K$ but not locally strongly internal to $K$.  The proof is, in essence, a restatement of  Johnson's results: 
\begin{lemma}\label{L: lasi=lsi=dim>0}
    Let $G$ be an infinite definable group in $\CK$. The following natural numbers are equal: 
    \begin{enumerate}
        \item $\dim(G)$,
        \item The $K$-rank of $G$,
        \item The almost $K$-rank of $G$.
    \end{enumerate}
    
    In particular, $G$ is locally strongly internal to $K$ if and only if it is locally almost strongly internal to $K$ if and only if $\dim(G)>0$.
\end{lemma}
\begin{proof}

    Since $\dim(X)\le \dim(G)$ for all $X\sub G$, and since dimension is preserved under definable finite-to-one functions, it is clear that $\dim(G)$ is an upper bound on the (almost) $K$-rank of $G$ (note that in $K^m$ dimension is the same as dp-rank). On the other hand, by \cite[Corollary 4.37]{JohnTopQp} there exists a definable local homeomorphism of some open (in the admissible topology) subset of $G$ with an open subset of $K^n$ for $n=\dim(G)$. This local homeomorphism witnesses that the $K$-rank of $G$ (and therefore also its almost $K$-rank) is, at least, $n$. The result follows. 
\end{proof}

\begin{remark}
 The equivalence of (1) and (2)  holds in certain expansions of real closed valued fields and algebraically closed valued fields, studied in \cite{HaHaPeGps}.  We postpone the proof to a subsequent paper.
\end{remark}

We now turn to showing that Johnson's admissible topology is the same as our $\tau_K$-topology. Note, before proceeding, that by the previous lemma, the $\tau_K$-topology is defined for any definable $G$ with $\dim(G)>0$, and as a matter of convention we may define $\tau_K$ to be the discrete topology if $\dim(G)=0$. 

\begin{lemma}\label{L:topologies coincides}
Let $G$ be an infinite definable group in $\CK$. Then $\tau_K$ coincides with the (unique) admissible $G$-topology.
\end{lemma}
\begin{proof}
    We may assume that $\dim(G)>0$, since otherwise both topologies are discrete.  Denote by $\tau_A$ the admissible topology. Since both topologies are group topologies, it is enough to show that they agree on a neighbourhood of a point.

    By \cite[Corollary 4.37]{JohnTopQp} there exists a definable subset $X\sub G$   and a $B$-definable $\tau_A$-homeomorphism $f:X\to U$ for some open $U\sub K^n$, where $n=\dim(G)$. If $d\in X$ is generic over $B$ then the collection of all definable $\tau_A$-open neighbourhoods of $d$ is exactly $\nu_X(d)$  (see \cite[Proposition 5.6]{HaHaPeVF}. By definition, $\nu_X(d)=\nu_K\cdot d$, and since the $\tau_K$-topology is right-invariant, the $\tau_K$-topology at $d$ equals the $\tau_A$-topology. Hence, the two topologies are equal.
\end{proof}

It follows from the above, together with \cite[Proposition 5.6]{HaHaPeGps} that $\nu_K$ is the infinitesimal subgroup of $G$, with respect to the admissible topology. I.e., it is the partial type given by the collection of all $K$-definable admissible neighbourhoods of $e$ (compare also to Lemma \ref{L: nu}). 

\section{Local analysis of definable groups}
We keep the convention that  $\CK=\CK^{\eq}$ is a sufficiently saturated elementary extension of a field $\mathbb{F}$ isomorphic to a finite extension of  $\qp$. Throughout, unless explicitly stated otherwise, we let $G$ be a $D$-group, for one of the distinguished sorts $D$. By this we mean, in particular, that $G$ is locally strongly internal to $D$ and that $D$ admits a type-definable subgroup $\nu_D$ as provided by \cite[Proposition 5.8]{HaHaPeGps} (see also Lemma \ref{L: nu}). In the proof of  Theorem \ref{T: intro1} and Theorem \ref{T: main into} we will reduce the general problem to the situation described above. 

This section is divided into subsections according to whether $D$ is $K$, $\Gamma$ or $K/\CO$. 
We show that if $G$ is locally strongly internal to either $\Gamma$ or $K/\CO$ then $G$ contains an infinite definable normal abelian subgroup (we also give a lower bound on its dp-rank). In the remaining case, where $G$ is $K$-pure, we endow $G$ with a $K$-differential structure and show  that $\ker(\ad)(G_1)\le  Z(G_1)$ for some definable normal subgroup $G_1\trianglelefteq G$ of finite index, where $\ad$ is the adjoint representation of $G_1$.

\subsection{Groups locally strongly internal to $K/\CO$}\label{ss: lsi to K/O}
Let $G$ and $\CK$ be as above and assume that $G$ is a $K/\CO$-group (so locally strongly internal to $K/\CO$).

\begin{fact}\label{F: torsion}
Let $\CK_0\equiv \CK$, $\CK_0$ not necessarily saturated. Then
\begin{enumerate}
    \item  $\mathrm{Tor}(K_0/\CO_0)=\{a\in K_0/\CO_0:v(a)\in \mathbb{Z}\}$.
    \item $\mathrm{Tor}(K_0/\CO_0)$ is a finite direct sum of Pr\"ufer $p$-groups and is isomorphic to $\mathbb{F}/\CO_{\mathbb{F}}$. In particular, $\mathrm{Tor}(K_0/\CO_0)$ is a $p$-group. 
    \item Every ball in $(K_0/\CO_0)^n$ centred at $0$ contains $\mathrm{Tor}(K_0/\CO_0)^n$ and the $p^k$-torsion points are 
    exactly the points $b\in (K/\CO)^n$ with $v(b)\geq -k$.

\end{enumerate}
\end{fact}
\begin{proof}
For our saturated $\CK$, clause (2) is \cite[Lemma 3.1]{HaHaPeGps}(3) (and the discussion preceding it), since in this situation we can embed $\mathbb{F}$, into $\CK$ and then $\mathrm{Tor}(\CK/\CO)=\mathbb{F}/\CO_{\mathbb{F}}$. Clause (1) follows from \cite[Lemma 3.1]{HaHaPeGps}(3) and clause (3) follows from the structure of the Pr\"ufer group.  It follows (from the basic properties of the Pr\"ufer group) that every proper subgroup (and in particular, the subgroup of $p^{k}$-torsion points) is a finite subgroup.  Thus, $\mathrm{Tor}(K/\CO)\sub \acl(\0)$, and because $\CK$ is saturated enough,  the results remain true in $\CK_0$. 
\end{proof}

\begin{lemma}\label{L:full subgroups have the same torsion}
Let $G$ be a definable $K/\CO$-group. Let $H_1, H_2\leq G$ be definable subgroups, and  $f_i:H_i\to (K/\CO)^n$ ($i=1,2$) definable group embeddings whose respective images are  open balls  in $(K/\CO)^n$, where $n$ is the $K/\CO$-rank of $G$. Then $\mathrm{Tor}(H_1)=\mathrm{Tor}(H_2)=f_1^{-1}(\mathbb{F}/\mathcal{O}_{\mathbb{F}})$ and $\dpr(H_1\cap H_2)=n$.
\end{lemma}
\begin{proof}
The assumptions and the conclusions are invariant under naming new constants, so we may assume that $\mathbb{F}$ is named in $\CK$ and so we may apply the results from \cite{HaHaPeGps}.

By the construction of $\nu_{K/\CO}$ (see Appendix B and Remark \ref{R: nu lives on any definable subgroup witnessing}) we have $\nu_{K/\CO}\vdash H_i$, $i=1,2$, hence   $\nu_{K/\CO}\vdash H_1\cap H_2$. By Lemma \ref{L: nu}(3), this implies that $\dpr(H_1\cap H_2)=n$. 

 Since $f_i(H_i)$ is an open ball, for $i=1,2$, it follows from  Fact \ref{F: torsion} that $\mathrm{Tor}(H_i)=f_i^{-1}((\mathbb{F}/\CO_{\mathbb{F}})^n)$. As $\dpr(H_1\cap H_2)=n$ also $\dpr(f_i(H_1\cap H_2))=n$ for $i=1,2$, so  by  \cite[Lemma 3.6]{HaHaPeGps} $f_i(H_1\cap H_2)$ has non-empty interior, thus contains a sub-ball of $(K/\CO)^n$. Therefore, (since it is a group) it also contains a ball centred at $0$. Thus, $(\mathbb F/\CO_{\mathbb F})^n\sub f_i(H_1\cap H_2)$ and hence $f_i^{-1}((\mathbb F/\CO_{\mathbb F})^n)\sub H_1\cap H_2$. We conclude \[\mathrm{Tor}(H_1)=f_1^{-1}((\mathbb F/\CO_{\mathbb F})^n)=f_2^{-1}((\mathbb F/\CO_{\mathbb F})^n)=\mathrm{Tor}(H_2),\] as needed.
\end{proof}

The next result implies that if $G$ is locally  $K/\CO$-strongly internal, then $G$ is not definably semisimple in any model over which $G$ is defined: 

\begin{proposition}\label{P:not semisimple K/O}
Let $\CK_0\prec \CK$ be an elementary substructure, $G$ a $\CK_0$-definable  $K/\CO$-group in $\CK$. Then there is a $K_0$-definable normal abelian $N\trianglelefteq G$ such that $\nu_{K/\CO}\vdash N$. In particular, $\dpr(N)$ is at least the $K/\CO$-rank of $G$. 
\end{proposition}
\begin{proof}
  As before, we may assume that $\mathbb{F}$ is named in $\CK$ and so we may apply the results of \cite{HaHaPeGps}. First, we show that an infinite normal abelian subgroup of $G$ is definable in $\CK$ and then we construct one that is $\CK_0$-definable as needed.
  
  By Lemma \ref{L: nu}(3) we can find a definable subgroup $H_0$,  $\nu_{K/\CO}\le H_0\leq G$,  that is definably isomorphic to a an open ball in $(K/\CO)^n$ centred at $0$, for $n$ the $K/\CO$-critical rank of $n$. Let $f:H_0\to (K/\CO)^n$ be a group embedding witnessing this.
  
  Let $H=\bigcap\limits_{g\in G} H_0^g$. It is a definable normal abelian subgroup. By Lemma \ref{L:full subgroups have the same torsion}, $\mathrm{Tor}(H_0^g)=\mathrm{Tor}(H_0^h)=f^{-1}((\mathbb F/\CO_{\mathbb F})^n)$, for any $g,h\in G$.  It follows, using compactness and saturation, that there is some  $r_0<\mathbb{Z}$ such that $B_{>r_0}(0)\sub f(H)$. 
  
  Assume that $H$ is definable over some $\widetilde t\in \CK$.
  Note that the statements that $H_{\widetilde t}$ is a normal abelian subgroup of $G$ and that $f_{\widetilde t}(H_{\widetilde t})$ contains a set of the form $B_{>r}(0)\sub (K/\CO)^n$ for some negative $r\in \Gamma$  (not excluding the case $r\in \Zz$) are first order in $\widetilde{t}$.
  
  This gives rise to a $\CK_0$-definable family of group embeddings $f_t:H_t \to (K/\CO)^n$, $t\in T$ 
  each $H_t$ normal abelian whose image in $(K/\CO)^n$ under $f_t$ contains $B_{>r}(0)$ for some $r\in \Gamma^{<0}$.  Let $\eta:T\to \Gamma$ be defined by $\eta(t)=\min\{r\in \Gamma:B_{>r}(0)\sub f(H_t)\}$.  By Lemma \ref{L:full subgroups have the same torsion}, for $t,s\in T$, if $\eta(t), \eta(s)<\mathbb{Z}$ then $\mathrm{Tor}(H_s)=\mathrm{Tor}(H_t)=f_t^{-1}((\mathbb F/\CO_{\mathbb F})^n)$. In particular, $\eta(\widetilde t)\leq r_0<\mathbb Z$, and therefore $\mathrm{Tor}(H_{\widetilde t})=f_{\widetilde t}^{-1}((\mathbb F/\CO_{\mathbb F})^n)$.
  
  Given a negative $r \in \Gamma$, let 
  \[
    G(r):=\bigcap \{H_t:\eta(t)\leq r\}.
   \] 
     
   Note that $G(r)$ is a definable normal abelian subgroup of $G$.  By our above observation, 
   for every $t$ such that $\eta(t)\leq r_0$, we have $f_{\widetilde t}^{-1}((\mathbb F/\CO_{\mathbb F})^n)\sub H_t$, thus $f_{\widetilde t}^{-1}((\mathbb F/\CO_{\mathbb F})^n)\sub G_{r_0}$. By compactness, there exists $r<\mathbb Z$ such that $f_{\widetilde t}^{-1}(B_{>r}(0))\sub G(r_0) $ and therefore   $\nu_{K/\CO}\vdash G(r_0)$ (by Lemma \ref{L: nu}).

  The family $G(r)$, as $r$ varies, is $K_0$-definable and increasing as $r$ tends to $-\infty$; the directed union $N:=\bigcup\limits_{r\in \Gamma_{<0}} G(r)$  is therefore  a $K_0$-definable normal, abelian and $\nu_{K/\CO}\vdash N$. Since the dp-rank of $\nu_{K/\CO}$ is the $K/\CO$-rank of $G$, the conclusion follows. 
\end{proof}

\subsection{Groups locally strongly internal to $\Gamma$}
In the present subsection we assume that $G$ is a $\Gamma$-group, so  locally strongly internal to $\Gamma$. We remind the following. 

\begin{fact}\label{F: def over Z}
For any definable family, $\{X_t\}_{t\in T}$, of subsets of $\Gamma^n$ the family $\{X_t\cap \Zz^n\}_{t\in T}$  is  definable in $\Zz_{Pres}$.
\end{fact}
\begin{proof}
Because $\CK$ is $p$-adically closed, $\Gamma$ is stably embedded. By a standard compactness argument, $\Gamma$ is uniformly  stably embedded, so we may assume that $T\sub \Gamma^k$ for some $k$. Since in  Presburger arithmetic types over $\mathbb{Z}$ are (uniformly) definable, the family is definable in $\Zz_{Pres}$. See  \cite[Theorem 0.7]{CoVo} (and also \cite{delon-def}).
\end{proof}

We note a few simple and useful lemmas.

\begin{lemma}\label{L:bound on index}
    Let $\{X_t:t\in T\}$ be a definable family of subsets of $\Gamma^n$ and assume that for all $t\in T$, $X_t\cap \mathbb{Z}^n$ contains a subgroup of $\mathbb{Z}^n$ of finite index. Then there is a uniform upper bound on $l(t)$,  the minimal $l\in \Nn$ such that $X_t\cap \Zz^n$ contains a subgroup $\Zz^n$ of index $l$.  
\end{lemma}
\begin{proof}
    Assume towards a contradiction  that there is no bound on $l(t)$ for $t\in T$.  So the following type is consistent:
    \[\rho(t):=\{D\not\subseteq X_t: D\subseteq \mathbb{Z}^n \text{ finite, generating a definable subgroup of finite index}\},\] contradicting the assumption.
\end{proof}

\begin{lemma}\label{F: full dp-rank, presburger}
\begin{enumerate}
    \item Let $Y\subseteq \Gamma^n$ be a definable subset. If $Y\cap \mathbb{Z}^n$ contains a subgroup of $\mathbb{Z}^n$ of finite index, then $\dpr(Y)=n$.
    \item Every finite index subgroup $H\leq \Gamma^n$ is definable. 
\end{enumerate}
\end{lemma}
\begin{proof}
By Fact \ref{F: def over Z}, $Y\cap \mathbb{Z}^n$ is definable in $\Zz_{Pres}$, as a subset of $\mathbb{Z}^n$. Since it contains a finite index subgroup, it has dp-rank $n$. So Clause (1) now follows by \cite[Lemma 3.10]{HaHaPeVF}. For Clause (2) note that since $H$ has finite index, there is $k\in \mathbb N$ such that $k(\Gamma^n)\leq H$, and then $H$ is a union of finitely many cosets of $k(\Gamma^n)$ so definable.
\end{proof}

\begin{lemma}\label{L: family of functions, presburger}
Let $Y\sub \Gamma^n$ be a definable set such that $Y\cap \Zz^n$ contains a subgroup $H$ of $\Zz^n$ of  finite index. Assume that $\{f_t\}_{t\in T}$ is a definable family of definable functions $f_t:Y \to Y$ whose restrictions to $H$ are group homomorphisms. Then: 

\begin{enumerate}
    \item For every $t\in T$, $f_t(H)\subseteq \mathbb{Z}^n$.
    \item The family $\{f_t\restriction H\}$ is uniformly definable in $\Zz_{Pres}$ and therefore finite.
\end{enumerate}
\end{lemma}
\begin{proof}
Assume everything is definable over some parameter set $A$. By stable embeddedness of $\Gamma$, the family $\{f_t: t\in T\}$ is uniformly definable in $\Gamma$ so we may assume that $T\sub \Gamma^k$. Since $H$ is a subgroup of finite index of $\mathbb{Z}^n$ it is generated by some finite set $\{m_1,\dots,m_s\}$.

(1) Fix some $t\in T$.  It will suffice to prove the claim for each coordinate function of $f_t$ separately. So we may assume $f_t:Y\to \Gamma$.  Let $c\in Y$ be $A$-generic in $Y$.

Since $\dpr(Y)=n$ it follows from  cell decomposition,  \cite[Theorem 1]{ClucPresburger}, and  \cite[Lemma 3.4] {OnVi} that there is an $A$-definable $n$-dimensional generalised box, $B=\prod_i J_i\sub Y$, centred at $c=(c_1,\ldots, c_n)\in B$,  such that 
\[
(f_t\restriction B)(x)=\sum_i s_i\left( \frac{x-t_i}{k_i}\right) +\gamma,
\] 
with  $\gamma\in \Gamma^n$, $s_i,t_i,k_i\in \mathbb N$  and $J_i=I_i\cap \{x-t_i\in P_{k_i}\}$, for some infinite interval $I_i$.


By shrinking $B$, if needed (over the same parameters), we may assume that $B$ is a product of boxes of the form $I_i\cap P_k(x_i-t_i)$ (i.e., that $k_i=k$ for all $i$). 

Note that $c+k\bar r\in B$  and that $f_t(c+k\bar r)-f_t(c)\in \mathbb{Z}$, for any $\bar r\in \mathbb{Z}^n$. In particular, if  $m_i$, $1\leq i\leq s$, is any of the above-mentioned generators of $H$, as $km_i\in H$ we conclude that $f_t(km_i)\in \mathbb{Z}$, but since $f_t(km_i)=kf_t(m_i)$ this implies that $f_t( m_i)\in \mathbb{Z}$ and, as this is true of a set of generators of $H$, we see that $f_t(H)\subseteq \mathbb{Z}$, as claimed.

(2) The first part of the claim  is a consequence of Fact \ref{F: def over Z} using Lemma \ref{F: full dp-rank, presburger}. The second part follows from quantifier elimination in Presburger arithmetic, by noting that  any definable family of group homomorphisms is finite (see also \cite[Fact 2.4 and Fact 2.6]{OnVi}). 
\end{proof}

The next results shows, in analogy with Proposition \ref{P:not semisimple K/O}, that groups strongly internal to $\Gamma$ are not semisimple:  
\begin{proposition}\label{P:not semisimple Gamma}
Let  $G$ be a definable infinite group in $\CK$. If $G$ is locally strongly internal to $\Gamma$ then there exists a definable $G_1\trianglelefteq G$ of finite index 
such that $\nu_{\Gamma} \vdash Z(G_1)$. In particular, $\dpr(Z(G_1)) $ is at least the $\Gamma$-rank of $G$. 
\end{proposition}
\begin{proof}
By Lemma \ref{L: nu} there are a definable subset $X\subseteq G$, with $\nu_{\Gamma}\vdash X$, and a definable function, $f:X\to \Gamma^n$, with $\dpr(X)=n$ for $n$ the $\Gamma$-rank of $G$. For simplicity of notation, identify $X$ with its image in $\Gamma^n$. We may further assume that $G$-multiplication coincides with addition and the same for inverse.
By Lemma  \ref{L: nu}, we may further assume that $\nu_{\Gamma}\vdash X$ is the intersection of generalized boxes around $0$. We fix one such generalized box $B\subseteq X\subseteq \Gamma^n$, $\nu_\Gamma\vdash B$. 
   
   By \cite[Proposition 5.8]{HaHaPeGps}, $g\nu_{\Gamma} g^{-1}=\nu_{\Gamma} $ for every $g\in G$ and thus $\nu_{\Gamma}\vdash B^g\cap B$. By compactness, for every $g\in G$, there exists a generalized box $B_0\subseteq B\cap B^g$ around $0$. As we noted above,   $B\cap \mathbb{Z}^n$ is a subgroup of $\mathbb{Z}^n$ of finite index (though $B^g$ need not be  contained in $\Gamma^n$).
    
    By Lemma \ref{L:bound on index} there is some natural number $k$ such that for any $g\in G$, $B^g\cap B$ contains a box $B_g$ with $B_g\cap \mathbb{Z}^n$ a subgroup of index at most $k$ in $\Zz^n$. Consequently, there exists some subgroup $H\subseteq \mathbb{Z}^n$ of finite index such that $H\subseteq B\cap B^g\cap \mathbb{Z}^n$ for all $g$.

    Let $Y=\bigcap\limits_{g\in G}B^g$. It is a definable set, contained in $B\subseteq \Gamma^n$, invariant under conjugation by all elements of $G$ and containing $H$.
     Let $Y_0:=Y\cap \mathbb{Z}^n$  (note that $H\subseteq Y_0$) and let $\tau_g:Y\to Y$  denote the restriction of conjugation by $g$ (in $G$) to $Y$. By Lemma \ref{L: family of functions, presburger}(1), $\tau_g(H)\subseteq \mathbb{Z}^n$. By Lemma \ref{L: family of functions, presburger}(2), $\{\tau_g\restriction H\}_{g\in G}$ is a  family of group homomorphisms uniformly definable in $\mathbb{Z}$, so it is finite. We may now replace $H$ by the (finite) intersection of all $\tau_g(H)$, and obtain another subgroup of finite index of $\mathbb Z^n$. Thus, we may assume that $H$ is invariant under all $\tau_g$.

    Let $R$ be a finite set of generators for $H$ and let $E(g,h)$ be the definable equivalence relation on $G$ given by $d^g=d^h$ for all $d\in R$. For any $g,h\in G$ both $\tau^g\restriction H$  and $\tau^h\restriction H$  respect addition,  hence $E(g,h)$ holds if and only if $\tau_g\restriction H=\tau_h\restriction H$. The  quotient $G/E$ can be identified with a finite subgroup of $\aut(H)$, and the map $\sigma:G \to G/E$ is a definable group homomorphism. Its kernel, $G_1$, is a definable normal subgroup of $G$ of finite index, which by definition centralizes $H$, hence $H\sub Z(G_1)$. We claim that $\nu_{\Gamma}\vdash Z(G_1)$.
    
    By Lemma \ref{F: full dp-rank, presburger}(2), $H$ is definable in $\Zz_{Pres}$ and  $Z(G_1)$ contains all finite boxes of the form $[-a,a]^n\cap H$, for $a\in \mathbb N$. Since $H$ is definable, $Z(G_1)$ must contain a set of the form $I^n\cap H(\CK)$, for an infinite interval $I\sub \Gamma$, so in particular, it contains a generalized box. It follows that $\nu_{\Gamma}\vdash Z(G_1)$.
\end{proof}

\subsection{$K$-groups}\label{Ss: K-gps}
Let $\CK$ and $G$  be as above, and assume that $G$ is a $K$-group, so locally strongly internal to $K$. As before $\mathbb F$ denotes an elementary submodel of $\CK$ isomorphic to a finite extension of $\qp$. 
In \cite[Theorem 7.11(1)]{HaHaPeGps} we have shown that if $G$ is locally almost strongly internal to $K$ then the associated type-definable infinitesimal subgroup $\nu_K$ can be endowed with a differential structure. In the present section, we develop a very rudimentary Lie theory for $\nu_K$ (equivalently, for local subgroups of $G$, strongly internal to $K$). We do not attempt to develop the theory from its foundations. Rather, working in $p$-adically closed fields, we transfer known results from $\qp$ (and its finite extensions) by working in definable families. This is our main reason for working in $p$-adically closed fields.

Recall the following (compare with \cite[Part II, Chapter IV.3]{serre-lie}):
\begin{definition} \label{D: loc gp}
Let $(K,v)$ be a valued field. A \emph{local group} is a tuple $\mathcal{G}=(X,Y,m,\iota,e;\varphi)$ such that 
\begin{itemize}
    \item The function $\phi: X\to K^n$ is a homeomorphism between the  topological space $X$ and  an open subset $U\sub K^n$;
    \item $Y\subseteq X$ is  open;
    \item $m:Y\times Y\to X$ and $\iota:Y\to Y$ are continuous functions;
    \item $e\in Y$ 
\end{itemize}
such that:
\begin{enumerate}
    \item For any $x\in Y$, $x=m(x,e)=m(e,x)$
    \item For any $x\in Y$, $e=m(x,\iota(x))=m(\iota(x),x)$.
    \item For any $Z\subseteq Y$ open subset containing $e$, with $m(Z\times Z)\subseteq Y$ and for all $x,y,z\in Z$, $m(x,m(y,z))=m(m(x,y),z)$.
\end{enumerate}

The local group $\mathcal{G}$ is differentiable (resp. $\CC^r$, analytic)  if  $\phi( m(\phi^{-1}(x), \phi^{-1}(y))$  and $\phi (\iota(\phi^{-1}(x))$  are differentiable (resp. $\CC^r$, analytic).

The local group  $\mathcal{G}$ is \emph{definable} in $\CK$, if all the objects, morphisms and $\varphi$ are $\CK$-definable. 
\end{definition}

Note that if  $(G,\cdot,  e)$ is a topological group then for any  $X\sub G$, an open neighbourhood of $e$, there exists and open $Y\sub X\sub G$ such that the pair $X,Y$  with $G$-multiplication and inverse restricted to $Y$ form a local group.  Such a pair $(X,Y)$ is \emph{a local subgroup of $G$}. 

\begin{definition}
    Let $\mathcal{G}=(X,Y,m, e, \iota;\varphi )$ and $\mathcal{G}'=(X',Y',m', e', \iota';\varphi' )$ be local groups. A \emph{local homomorphism} $f:\mathcal{G}\to\mathcal{G}'$  is a continuous function $f: U\to X'$, where $U$ is an open neighbourhood of $e$, such that $f(e)=e'$ and $f(m(x,y))=m(f(x), f(y))$ in a neighbourhood of $e$.    If $\CG, \CG'$ are differentiable (resp. $C^n$, analytic) local groups, then such an $f$ is \emph{differentiable} (resp. $\CC^r$, analytic) if $\varphi'\circ f\circ \varphi^{-1}$ is differentiable (resp. $\CC^r$, analyitic). 
\end{definition}

The following statement and its proof are similar to other results of the same nature, see for example \cite[Lemma 3.8]{PilQp} and the main result of  \cite{Pi5}) for details, so we are terse.

\begin{lemma} \label{lemma analytic} Every definable local $C^k$-group in $\mathbb F$ is definably locally $C^k$-isomorphic to a definable local analytic group. 
\end{lemma}
    
    \begin{proof}
         For the sake of clarity, we use standard multiplicative notation to denote the local group operations.  So let  $\mathcal{G}=(X,Y,\cdot ,  ^{-1},1;\varphi )$ be a definable local $C^k$-group in $\mathbb F$. Changing names, we may assume that  $\phi=\id$ and $X$ is an open subset of $\mathbb F^n$.
    
        By \cite[Lemma 1.3]{Sc-vdDries}, if $U\sub \mathbb F^n$ is open, then every definable function $f:U\to V$ in $\mathbb Q_p$ is analytic on a  definable large open $U_0\subseteq U$ (i.e.,  $\dpr(U\setminus U_0)<\dim(U)$). This remains true in  finite extensions of $\mathbb{Q}_p$ by \cite[Section 5]{Sc-vdDries}.  Using this fact, we can apply (the proof of)   \cite[Lemma 4.38]{HaHaPeGps}  to find an open definable subset $Y_0\subseteq Y$ such that the function $F(x,y,z):=xy^{-1}z$  is defined and analytic on $Y_0^3$. Fix $a\in Y_0$, and let $\hat \phi:Y_0a^{-1}\to Y_0$ be given by $x\mapsto x\cdot a$. We now let $X_1=Y_0a^{-1}$ and $Y_1=(Y\cap X_1)\cap (Y\cap X_1)^{-1}$, both are open definable sets. It is easy to verify that $\CG'=(X_1,Y_1,\cdot ,^{-1},e;\hat \phi)$ is a definable analytic local group. Obviously, the identity map between $\CG'$ and $\CG$ is a local homomorphism between local groups. Since definable functions are $C^k$ on large sets, this identity map is a $C^k$-isomorphism from $\CG$ to $\CG'$.
    \end{proof}
    
The next lemma is the main Lie theory result  we need. See a remark after the proof.

\begin{lemma}\label{L:bourbaki for local groups}
        Let $\mathcal{G}=(X,Y,\cdot, ^{-1},e;\varphi)$ be a definable differentiable local group for $X\sub K^n$,  and $f: Y\to X$ a differentiable local group homomorphism. If $D_e(f)=\id$ then $\{y\in Y:f(y)=y\}$ is an open subset of $Y$ containing $e$. 
\end{lemma}
\begin{proof}
    We prove that $\{y\in Y:f(y)=y\}$ contains a definable open subset, which is enough to conclude. Let $\mathbb{F}$ be a finite extension of $\mathbb{Q}_p$, embedded into $K$. 
    Since the lemma is first order in families, it will suffice to prove it in $\mathbb{F}$. So fix  a local group $\mathcal{G}=(X,Y,\cdot, ^{-1},e;\varphi)$ and a local group homomorphism $f:Y\to X$ all as in the statement. There is no harm in assuming that $\varphi=\id$, i.e. that $Y\subseteq X\subseteq \mathbb{F}^n$ are open balls around $e$.

    By Lemma \ref{lemma analytic}, we may assume, without loss of generality,  that $\CG$ is a definable analytic local group .
     By \cite[Part II, Chapter IV.8]{serre-lie} (see also \cite[Lemma 1.14]{helge-lectures}), there exists some $\gamma\in \Gamma_{\mathbb{F}}$, such that $B_{>\gamma}(e)\subseteq Y_0$ and that $(B_{>\gamma}(e), m,\iota,e)$ is an analytic Lie subgroup of $\mathcal{G}$, and furthermore the same is true for any $\gamma_0>\gamma$. By replacing $\mathcal{G}$  with $B_{>\gamma}(e)$, we may thus assume that $\mathcal{G}$ is an analytic Lie group (rather than a local one).

    We may now apply the inverse function theorem for analytic functions, \cite[Lemma 1.11]{helge-lectures}. Since $D_e(f)=\id$, there is some $\gamma_0>\gamma$ such that $f(B_{>\gamma_0}(e))=B_{>\gamma_0}(e)$ and that $f\restriction B_{>\gamma_0}(e)$ is injective.
    So now $B_{>\gamma_0}(e)$ is an analytic Lie group and $f$ an analytic automorphism thereof.
    
    By \cite[Chapter 3.3.8, Corollary of Proposition 29] {bourbaki-lie}, $W:=\{y\in B_{>\gamma_0}(e): f(y)=y\}$ is a sub Lie group of $B_{>\gamma_0}(e)$ with the same Lie algebra. It now follows from the theory of $p$-adic Lie groups  (e.g.,  \cite[Part II, Chapter III.9]{serre-lie}) that $W$ is open. It follows that $\{y\in Y:f(y)=y\}$ contains an open neighbourhood of $e$ in $Y$, as it is a topological group it implies that it is open as well. 
\end{proof}

\begin{remark} \label{R: general Lie}
The  literature on $p$-adic Lie groups known to us is restricted,  almost solely,  to analytic groups. Thus, in the above argument, we had to work over $\mathbb F$, obtain an analytic local group, use the classical theory of analytic Lie groups to obtain the first order statement we needed, and transfer this statement to arbitrary $p$-adically closed fields. It will be desirable to develop the necessary theory for definable differentiable functions in expansions of $p$-adically closed fields, under the appropriate assumptions. 
\end{remark}

\begin{lemma}\label{L:existence of local group}
    Let $(G,\cdot,^{-1},e)$ be a definable group of positive dimension, $W\ni e$ an open subset of $G$. Then there exists a definable differentiable local subgroup $(X,Y,\cdot, ^{-1},e;\varphi)$ with $X\sub W$ an open subset and $\varphi: X\to U$  a definable bijection, where $U\subseteq K^n$ is a ball.
    
    Moreover, for any $g\in G$ the map $\tau_g: x\mapsto x^g$ induces on $X$ a definable differentiable local group isomorphism.
\end{lemma}
\begin{proof}
     By definition, $\nu_K\vdash W$, so by Lemma \ref{L: nu}(1), there exists some definable open subset $X\subseteq W$ and a definable injection $\varphi:  X\to K^n$, with $\varphi(X)$ an open ball and $n$ the $K$-rank of $G$.  
     
     Let $\widehat \CK\succ \CK$ be a $|\CK|^+$ saturated elementary extension. By \cite[Theorem 7.11(a)]{HaHaPeGps}, $\nu_K(\widehat \CK)$ is a (differentiable) Lie group with respect to the structure induced  by $\varphi$. Furthermore, $g\nu_K g^{-1}=\nu_K$ for any $g\in G(\CK)$ (\cite[Proposition 5.8(3)]{HaHaPeGps}). Since definable functions in $\CK$ are generically differentiable (\cite[Lemma 1.3]{Sc-vdDries}), and the conjugation map $\tau_g$, for $g\in G(\CK)$, is a group homomorphism, $\tau_g$ is a differentiable homomorphism of $\nu_K(\widehat \CK)$ (see also \cite[Proposition 4.19(2)]{HaHaPeVF} for a similar argument). We can now find the desired subset $Y\subseteq X$ by compactness. Replacing $Y$ with $Y\cap Y^{-1}$ we obtain our desired local group.
\end{proof}

The above lemma allows us to define the adjoint representation of definable groups: 
\begin{notation}\label{N:Ad}
Let  $(G,\cdot,^{-1},e)$ be a definable group of positive dimension,   $\mathcal{G}=(X,Y,\cdot,^{-1};\varphi)$  a local subgroup, as provided by Lemma \ref{L:existence of local group}. For $g\in G$, we let  $\ad(g):=D_e(\tau_g)\in M_n(K)$, the Jacobian matrix of the map $\tau_g:x\mapsto x^g$ (as a map from $Y$ to $X$) at $e$.
\end{notation}

By the chain rule, $\ad$ is a homomorphism of groups. 
Note that while the matrix $D_e(\tau_g)$ may depend on the choice of $\phi$ (up to conjugation), the definable group  $\ker(\ad)$ does not. From now on, we use $\ker(\ad)$ without  mention of the differentiable local group in the background. \\

The following is needed in the sequel.
\begin{remark}\label{R:U subset of kerAd}
If $U\sub G$  is an open neighbourhood of $e$ such that  $xy=yx$ for all $x,y\in U$ then, by Lemma \ref{L:existence of local group}, there exists a local differentiable abelian subgroup of $G$ (contained in $U$). As $\tau_g\rest U=\id$ for all $g\in U$, we get that $U\subseteq \ker(\ad)$. 
\end{remark}

We can now apply  Lemma \ref{L:existence of local group}.

\begin{proposition}\label{P:dim C_G(g)=dimG}
    Let $G$ be a definable group of positive dimension. If  $g\in \ker(\ad)$ then $\dim C_G(g)=\dim G$.
\end{proposition}
\begin{proof}
    Let $\mathcal{G}=(X,Y,\cdot,^{-1};\varphi)$ be the definable differentiable local group as provided by Lemma \ref{L:existence of local group}. If $g\in \ker(\ad)$ then by Lemma \ref{L:bourbaki for local groups}, we get that  $W:=\{x\in Y: x^g=x\}$ is open in $Y$. Since $\dim (Y)$ is  the $K$-rank of $G$ (Lemma \ref{L: lasi=lsi=dim>0}), we get that
    \[
    \dim (G)=\dim(Y)=\dim (W)\leq \dim C_G(g)\leq \dim(G)
    ,\] as required.
\end{proof}

Recall that a definable group $G$ is \emph{$K$-pure} if $G$ is locally  strongly internal to $K$ but not locally almost strongly internal to $\Gamma$ or to $K/\CO$. The following is based on an analogous result of Gismatullin, Halupczok and Macpherson: 

\begin{corollary}\label{C:dugaldetal}
    Let $G$ be a definable group, locally strongly internal to $K$ and let $g\in G$. If $G$ is $K$-pure and $\dim(C_G(g))=\dim(G)$ then $[G:C_G(g)]<\infty$. In particular, $[G:C_G(g)]<\infty$ for every $g\in \ker(\ad)$. 
\end{corollary}
\begin{proof}
    The conjugacy class is in definable bijection with the imaginary sort $G/C_G(g)$. By additivity of dimension we get that $\dim(g^G)=\dim(G)-\dim(C_G(g))$. If $\dim(C_G(g))=\dim(G)$ then  $\dim(g^G)=0$.
        By Lemma \ref{l: pure 0dim}, $g^G$ is finite hence $[G:C_G(g)]$ is finite. \end{proof}

Before the next corollary we note the following application of Baldwin-Saxl.

\begin{fact}\label{Baldwin Saxl} 
Let $G$ be a definable group in $\CK$. Assume that $X\sub G$ is a definable set and for every $a\in X$, $C_G(a)$ has finite index in $G$. Then there exists a definable normal subgroup $G_1\trianglelefteq G$ of finite index with $G_1\leq C_G(X)$.
\end{fact}
\begin{proof}
Since $\CK$ is an NIP structure, we may apply Baldwin-Saxl (e.g., \cite[Lemma 1.3]{PoiGroups}) to conclude that there is a a finite bound on the index of finite intersections of subgroups of the form $C_G(a)$, $a\in X$.  Consequently, $C_G(X)=\bigcap \{C_G(a): a\in X\}=C_G(a_1)\cap \dots \cap C_G(a_m)$, for some $a_1,\dots, a_m\in X$. In particular, $C_G(X)$ has finite index in $G$. By general group theory,  the intersection of all the conjugates of $C_G(X)$, call it $G_1$, is a normal subgroup of $G$  of finite index. 
\end{proof} 
As a  corollary, we get: 

\begin{corollary}\label{C: ZG=ker}
    Let $G$ be a definable $K$-pure group. Then $C_G(\ker(\ad))$ has finite index in $G$. Moreover, there exists a finite index normal (open) subgroup $G_1\trianglelefteq G$ such that 
    $\ker(\mathrm{Ad}\restriction G_1)\leq Z(G_1)$.
\end{corollary}
\begin{proof} 
By Corollary \ref{C:dugaldetal}, for every $g\in \ker(\ad)$, $C_G(g)$ has finite index in $G$. By Fact \ref{Baldwin Saxl}, there exists a normal definable subgroup $G_1\trianglelefteq G$ such that  $G_1\le C_G(\ker(\ad))$ so $C_G(\ker(\ad))$, too, has finite index. It follows, also, that $\ker(\ad\restriction G_1)=\ker(\ad)\cap G_1\leq Z(G_1)$. Note that Since $G_1$ has finite index in $G$ it is open.
\end{proof}

\begin{remark}
In the notation of Corollary \ref{C: ZG=ker}, note that if $G_1\leq G$ is an open subgroup, then $\ad\restriction G_1=\ad_{G_1}$.
\end{remark}

The above result implies that for $K$-pure groups, the kernel of $\ad$ has an open normal abelian subgroup of finite index. This is true in particular for  $p$-adic Lie groups definable in the $p$-adic field. Recently,  \cite{helge-example},  Gl\"ockner constructed an example of a $1$-dimensional $p$-adic Lie group $G$ for which this fails. In fact, in his example $\ker (\ad)=G$, but $G$ contains no open normal abelian subgroup.

\section{The Main results}
Recall that a group is \emph{definably simple} if it has no definable normal subgroups, and  \emph{definably semisimple} if it has no  definable infinite normal abelian subgroups. 
The main goal of the present section is to show that definably semisimple groups are, after quotienting by a finite  normal subgroup,  definably isomorphic to $K$-linear groups. 
As a corollary, we show that the notion of a definably semisimple group is elementary, despite the fact that $K^{eq}$ does not eliminate the quantifier $\exists^\infty$. I.e., if $\CK_0\prec \CK$ and $G$  is a $K_0$-definable group, such that $G$ is     definably semisimple in $K_0$ then it remains so in $K$.

As before, $\CK=\CK^{eq}$ is a sufficiently saturated $p$-adically closed field. Throughout the previous section we were working under the assumption that our definable group $G$ is a $D$-group (for some distinguished sort $D$).  As shown in \cite{HaHaPeGps}, this need not be the case. The best we can obtain, in general, is that if $G$ is locally \emph{almost} strongly internal to $D$  then there is a finite normal subgroup $H$ such that $G/H$ is a $D$-group (so in particular, locally strongly internal to $D$), \cite[Proposition 4.35]{HaHaPeGps}. Our first order of business is to verify that taking such quotients is, essentially,  harmless. 

\subsection{Some group theoretic facts}

We need a couple of group theoretic observations on definable groups in our setting. We note for future reference that in both Lemma \ref{L:groups 1} and Corollary \ref{C: semisimple in quotients} below saturation of $\CK$ is not used: 
\begin{lemma}\label{L:groups 1}
Let $N$ be a definable group in $\CK$ and $H\trianglelefteq N$ a definable normal subgroup, such that $N/H$ is abelian. 
For $k\in \mathbb N$, let $N^k=\{g^k:g\in N\}$. Then: 
\begin{enumerate}
    \item For every $k\in \mathbb N$, $N^k H$ is a normal subgroup of $N$ and $N/N^kH$ is finite.
    
    \item If $H$ is finite and central,  and $k=|H|$ then $N^k\sub Z(N)$ and $Z(N)$ has finite index in $N$.
    
\end{enumerate}
\end{lemma}

\begin{proof} (1) Since $N/H$ is abelian, for every $a,b\in N$, $ab=bah$ for some $h\in H$. Because $H$ is normal, for all $g\in G$ and $h\in H$ there is $h'\in H$ such that $hg=gh'$. It follows that $a^2b^2=(ab)^2h_1$, for $h_1\in H$, and by induction, $a^kb^k=(ab)^k h_0$, for some $h_0\in H$. Thus $N^kH$ is a subgroup, clearly normal in $N$.

The order of every  $g\in N/N^kH$ is at most $k$, thus $N/N^kH$ has bounded exponent.
The group $N/N^kH$ is clearly also definable in $\CK$, and in  \cite[Theorem 7.11]{HaHaPeGps} we showed that every infinite such group must have unbounded exponent.
Thus, $N/N^kH$ must be finite.

(2) Assume now that $k=|H|$ and $H$ is central. Since $G/H$ is abelian,  For every $g, x\in N$ we have $g^{-1}xg=x h$ for some $h\in H$, and hence, since $H$ is central,  $g^{-1}x^k g=(xh)^k=x^kh^k=x^k$. Thus $N^k\sub Z(N)$. It follows that $N^kH\sub Z(N)$, so by (1), $Z(N)$ has finite index in $N$.
\end{proof}

The proof of the next corollary is simpler when $H$ is central, but for our needs we have to avoid this assumption.

\begin{corollary}\label{C: semisimple in quotients}
 Let $G$ be a definable group in $\CK$ and $H$ a finite normal subgroup of $G$.
 If $G/H$ contains a definable normal abelian subgroup of dp-rank $k$ then  so does  $G$. In particular, if $G$ is definably semisimple, then so is $G/H$.
\end{corollary}
\begin{proof} 
 By Lemma \ref{Baldwin Saxl}, there exists a definable $G_1\trianglelefteq G$ of finite index such that $G_1\sub C_G(H)$. In particular, $G_1\cap H$ is central in $G_1$.

 Assume that $G/H$ has an infinite definable abelian normal subgroup of the form $N/H$ for $N\trianglelefteq G$. Let  $N_1:=N\cap G_1$, an infinite normal subgroup of $G$ of finite index in $N$ and $H_1:=H\cap N_1$, a central subgroup of $N_1$. The quotient $N_1/H_1$ is isomorphic to $N_1H/H\sub N/H$ so is  abelian. 

By Lemma \ref{L:groups 1} (2), $Z(N_1)$ has finite index in $N_1$ and therefore $\dpr(Z(N_1))=\dpr(N_1)=\dpr(N)=\dpr(N/H)$. Because $N_1$ is normal in $G$ so is $Z(N_1)$. Hence, $Z(N_1)$ is a definable  normal abelian subgroup of $G$ of the same rank as $N_1/H$.
\end{proof}

\subsection{The main results}
In this section we prove the main results of the paper: those claimed in the introduction, as well as some corollaries. Recall from \cite[\S 9.3]{JohnTopQp} that a definable group $G$ is \emph{locally abelian} if there exists $W\ni e$, a neighbourhood of $e$ in $G$ (in the admissible topology) such that $xy=yx$ for all $x,y\in W$. The following is immediate:
\begin{lemma} 
A definable group $G$ in $\CK$ of positive dimension is locally abelian if and only if $\nu_K$ is abelian (meaning that $\nu_K(\CK')$ is an abelian group in every $|K|^+$-saturated $\CK'\succ \CK$).
\end{lemma}

The next theorem gives conditions under which a definable,  infinite, abelian normal subgroup exists in $G$. 

\begin{theorem}\label{T:johnson question}
    Let $G$ be an infinite group definable over a  $p$-adically closed field $\CK_0\prec \CK$. 
    \begin{enumerate}
        \item If $G$ is $K$-pure and locally abelian then there exists a definable abelian subgroup  $G_1\trianglelefteq G$ of finite index, defined over $\CK_0$. In particular, $G_1$ is open. 
        
        \item If $G$ is not $K$-pure  then there exists a $\CK_0$-definable infinite normal abelian subgroup $N\trianglelefteq G$.
        More precisely, if $G$ is locally almost strongly internal to $\Gamma$ or to $K/\CO$ then $\dpr(N)$ is greater or equal to  the almost $\Gamma$-rank of $G$, or the almost $K/\CO$-rank of $G$, respectively.

    \end{enumerate}
    
    In any case, if $G$ is locally abelian then it contains a definable infinite normal abelian subgroup.

\end{theorem}

\begin{proof}
(1) 
     Assume first that $G$ is locally abelian and $K$-pure. The assumption implies that $\dim(G)>0$, hence by Lemma \ref{L: lasi=lsi=dim>0}, $G$ is locally strongly internal to $K$. 
     Since $G$ is locally abelian, by Remark \ref{R:U subset of kerAd} we can find a definable open subset $U\subseteq \ker(\ad)$, this gives $\dim(\ker(\ad))=\dim(G)$.   
     
     The proof that $G$ is abelian-by-finite is an  adaptation of \cite[Proposition 2.3]{PiYao}. By Corollary \ref{C:dugaldetal} $[G:C_G(a)]<\infty$ for all $a\in U$. Since $U$ is definable, by compactness and saturation, there is a uniform bound  on $[G:C_G(a)]$. By Fact \ref{Baldwin Saxl}, there is a definable normal subgroup of finite index $H_0\trianglelefteq G$ such that $H_0\leq C_G(U)$. 
     
     For every $h\in H_0$, $U\sub C_G(h)$ hence $\dim C_G(h)=\dim G$,
     By  Corollary \ref{C:dugaldetal} and $K$-purity, we have $[G:C_G(h)]<\infty$ for every $h\in H_0$. Thus, applying Fact \ref{Baldwin Saxl} again, we see that  $C_G(H_0)$ has finite index in $G$, so in particular,  $G_1=C_G(H_0)\cap H_0$ has finite index in $G$ and is commutative. It follows that $G_1$ is open by \cite[Proposition 5.18]{JohnTopQp}. The fact that $G_1$ is a definable, open, normal abelian, subgroup of index $k$ (some $k\in \Nn)$, is first order, so we can find such $G_1$ defined over  $\CK_0$. 

    (2) Assume now that $G$ is not $K$-pure. 
    By \cite[Lemma 7.10]{HaHaPeGps}, $G$ is locally almost strongly internal to $D=\Gamma$ or $D=K/\CO$. By  \cite[Proposition 4.35]{HaHaPeGps} there exists $H\trianglelefteq G$ a finite normal subgroup such that  $G/H$ is locally strongly internal to $D$ and a $D$-group. Moreover, the $D$-rank of $G/H$ is the almost $D$-rank of $G$. By \cite[Proposition 4.35(2)]{HaHaPeGps} $H$ is invariant under all the automorphisms of $\CK$ over $\CK_0$. Because $H$ is definable, this means that $H$ is $\CK_0$-definable.

    Assume  that $D=\Gamma$. By Proposition  \ref{P:not semisimple Gamma}, we have $\nu_{\Gamma}(G/H)\vdash Z(G/H)$. In particular, $G/H$ contains a normal abelian subgroup whose dp-rank is at least the $\Gamma$-rank of $G/H$ (equivalently, the almost $\Gamma$-rank of $G$). By Corollary \ref{C: semisimple in quotients}, $G$ contains a definable normal abelian subgroup of the same dp-rank.
    
    Finally, assume that $G$ is locally almost strongly internal to $K/\CO$, so $G/H$ is locally strongly internal to $K/\CO$ and its $K/\CO$-rank equals the almost  $K/\CO$-rank of $G$. By Proposition \ref{P:not semisimple K/O}, as $G$ and $H$ are both $\CK_0$-definable, there exists a $\CK_0$-definable infinite normal abelian subgroup of  $G/H$ whose dp-rank is at least the almost $\Gamma$-rank of $G/H$
    By Corollary \ref{C: semisimple in quotients}, $G$ contains a definable normal abelian group of the same rank.
\end{proof}

The following example shows that the assumption of $K$-purity is needed in Theorem \ref{T:johnson question}, in order to find in a definable locally abelian group, a definable open normal abelian subgroup:

\begin{example}\label{E:counter}
Let $\CO^\times$ denote the multiplicative group of $\CO$. Consider the semi-direct product $G=\CO^\times \ltimes K/\CO$, where $(a,b+\CO)\cdot (c,d+\CO)=(ac, b+ad+\CO)$. Then $\dim(G)=1$ and $\dpr(G)=2$. It is locally abelian, as witnessed by $\CO^\times\times \{0\}$. We claim that $G$ has no definable open normal abelian subgroup. Assume, towards a contradiction, that  $H$ is such, in particular by \cite[Theorem 1.4(1)]{JohnTopQp} $\dim(H)=\dim(G)$ so $\pi_1(H)$, the projection on the first coordinate, must be infinite.

Let $(t,0)\in H$ for $t\neq 1$. Since the conjugation of $(t,0)$ by $(a,b+\CO)$ is $(t,b-bt+\CO)$, by letting $b$ vary we conclude that $\pi_2(H)$, the projection on the second coordinate, is equal to $K/\CO$. Thus, $H=U\ltimes K/\CO$ for some infinite definable subgroup $U$ of $\CO^\times$. Every element of $\CO^{\times}$ acts non-trivially on $K/\CO$, thus $U\ltimes K/\CO$ is not abelian unless $U=\{1\}$,  proving that $H$ as required does not exist. 

On the other hand, note that $\{1\}\times K/\CO$ is an infinite definable normal abelian subgroup (that is not open).
\end{example}

Theorem \ref{T:johnson question} together with the above example answer a question of Johnson's \cite[\S 9.3]{JohnTopQp}. As a corollary of Theorem \ref{T:johnson question} we obtain Theorem \ref{T: intro1}: 

\begin{theorem}\label{T:main}
Let $\CK_0$ be a  $p$-adically closed field  and let $G$ be a definable,  definably semisimple group in $\CK_0$. Then there exist  $H_0\le G$ a finite normal subgroup, defined over the same parameters as $G$, such that  $G/H_0$
is definably isomorphic to a definable $K_0$-linear group. 
\end{theorem}
\begin{proof}
    Let $\CK\succ \CK_0$ be a sufficiently saturated extension. By Theorem \ref{T:johnson question} (2), $G(\CK)$ must be  $K$-pure for otherwise it has a definable infinite normal abelian subgroup.  
    
    We consider $\ad:G\to \gl_n(K)$ and claim that $\ker(\ad)$ must be finite.
    Indeed, by Corollary \ref{C: ZG=ker}, there exists a  definable  (finite index) subgroup $G_1\trianglelefteq G$ such that $\ker(\ad)\cap G_1\sub Z(G_1(K))$. In particular, $\ker(\ad)\cap G_1$ is abelian. It is normal as the intersection of two normal groups,  hence, by our assumptions, must be finite. Since $G_1$ has finite index in $G$ it follows that $\ker(\ad)$ is finite, as claimed. Finally, the group $G/\ker(\ad)$ is $K$-linear and the homomorphism is defined over $K_0$.
\end{proof}

We isolate from the proof the following useful observation: 
\begin{lemma}\label{L: finite ker}
    Under the same assumptions, if $G$ is a definable, definably semisimple group then $G$ is  $K$-pure and $\ker(\ad)$ is finite. 
\end{lemma}

As a special case we get: 
\begin{corollary}
    If a group $G$, definable in a $p$-adically closed field $\CK_0$,  is definably simple then it is definably isomorphic to a $K_0$-linear group. 
\end{corollary}

We also have the following.

\begin{corollary}\label{C: ss is fo}
    Let $\CK_0$ be a $p$-adically closed field and $G$ a $K_0$-definable group. Then $G(K_0)$ is definably semisimple  if and only if $G(K)$ is.
\end{corollary}
\begin{proof}
    
    If $G(K)$ is definably semisimple, then so is $G(K_0)$.  So we assume that $G(K_0)$ is  definably semisimple and show that so is $G(K)$.

      By Theorem \ref{T:johnson question}(2),  $G$ is $K$-pure; so by Theorem \ref{T:main}, there exists a finite normal subgroup $H_0\trianglelefteq G$ with $G/H_0 (\CK_0)$ definably isomorphic to a $K_0$-linear group. Note that $G/H_0(\CK_0)$ is  definably semisimple by Corollary \ref{C: semisimple in quotients}. As $\CK_0$ has uniform finiteness for definable families of definable subset of $K_0^m$, for any $m$, it follows that $G/H_0(\CK)$ is definably semisimple as well. However, since $H_0$ is finite,  $G(\CK)$ is definably semisimple. 
\end{proof}

The above argument shows in particular:
\begin{corollary}
 Let $\CK_0$ be a $p$-adically closed field and $G$ a $K_0$-definable group.  If $G$ is definably semisimple, then so is $G/\ker(\ad)$.
\end{corollary}

\subsection{Possible generalisations}
The results presented in the present paper build heavily on our earlier works, \cite{HaHaPeGps} and \cite{HaHaPeVF} which were carried out in the context of $P$-minimal, $1$-h-minimal valued fields with definable Skolem functions. Many of the results presented here can be proven in these more general settings, but some obstacles to proving the theorem about definably semisimple groups in these settings do remain:

\begin{enumerate}
    \item The main obstacle to proving our results in  $P$-minimal expansions of $p$-adically closed fields is the one described in Remark \ref{R: general Lie}, namely the lack of the basic Lie Theory in this  setting.

    \item Many of our results can be extended to real closed valued fields and to algebraically closed fields of equi-characteristic $0$. The proof of Theorem \ref{T:johnson question}(1)  goes through almost unaltered, and an analogous  result exists  for the residue field $\bk$. Theorem \ref{T: main into}(3) is also valid in this context (but the proof is significantly different). These will appear in a subsequent paper. We do not, however, know how to extend our results  on $K/\CO$-groups to these settings. 
    The main remaining  obstacle is proving an analogue of Proposition \ref{P:not semisimple K/O} but we conjecture that such a statement is true: If $G$ is locally strongly internal to $K/\CO$ then there exists an infinite, normal, abelian subgroup $N\trianglelefteq G$ such that $\nu_{K/\CO}\vdash N$.
   
\end{enumerate}

\appendix

\section{Coordinate Projections in  $(K/\CO)^n$}\label{A:appendix}

In this section of the appendix we prove a technical result we needed in several places in the paper. Let $(K,v,\dots)$ be a finite dp-rank expansion of a $p$-adically closed.

\begin{lemma}\label{L:local homeo K/O}
If $H\leq (K/\CO)^n$ is a definable infinite subgroup, then there is a subgroup $H_1\leq H$, with $\dpr( H_1)=\dpr (H)=d$ such that $H_1$ projects invectively into some $(K/\CO)^{d}$.
\end{lemma}
\begin{proof}

First, note that every finite non-trivial subgroup of $K/\CO$ contains $G_{-1}:=\{g\in K/\CO:v(g)\leq -1\}$. Indeed, all torsion elements of $K/\CO$ are of order $p^m$, for some natural number $m$, and thus every finite group must contain $C_p$, the cyclic group of order $p$. The group of elements of order $p$ is exactly $G_{-1}$.

We use induction on $n$, where the case $n=1$ is trivially true.
Assume that $H\subseteq (K/\CO)^n$. For $i=1,\ldots, n$, let $\pi^i:(K/\CO)^n\to (K/\CO)^{n-1}$ denote the projection onto the remaining $n-1$ coordinates. Let $H^i=\ker(\pi^i\restriction H)$.

Notice that the group product of the $H^i$ in $(K/\CO)^n$ is a direct product, thus if all $H^i$ are infinite then $\dpr(H)=n$ and we are done. Hence, one of the $H^i$'s must be finite.
We use additional induction on $\min_i |H^i|$.

If this minimum is $1$, then one of these projections $\pi^i$ is injective on $H$ and hence it is sufficient to prove the statement for $\pi^i(H)$, and then we may finish by induction. So we may assume that $1<\min_i |H^i|<\infty$.

Assume next that all $H_i$ are finite (and non-trivial). Then they all contain $G_{-1}$.
Consider the (surjective) definable map \[\sigma:(K/\CO)^n\to (K/\CO)^n\,\, \mbox{ defined by } \sigma(x)=p\cdot x.\]
By our assumption, $(G_{-1})^n= \ker(\sigma)\sub H$.  

Let $N=\sigma(H)\sub (K/O)^n$ and $N^i:=\ker(\pi^i\rest N)$.

\begin{claim}
	For every $i$, $|N^i|<|H^i|$.
\end{claim} 

\begin{claimproof}

In fact, we shall see that $\sigma(H^i)=N^i$, with non-trivial kernel $G_{-1}$, so the result follows
Let us see that for $i=1$. Clearly, if $(x,0,\ldots,0)\in H^1$ then $(px,0,\ldots,0)\in N^i$.
Conversely, if $y=(y_1,0,\ldots, 0)\in N^1$ then there is $x=(x_1,x_2,\ldots, x_n)\in H$, such that $\sigma(x)=y$. This implies that $px_1=y_1$ and $x_2,\ldots, x_n\in G_{-1}$.
Because $G_{-1}^n\sub H$ it follows that $x'=(x_1,0,\ldots,0)$ is also in $H$, so in $H^1$,  and we have $\sigma(x')=y$ as we wanted, thus proving our claim. 
\end{claimproof}

We can now apply induction to the group $N$ and conclude that there is $N_1\sub N\sub (K/\CO)^n$ of the same dimension as $N$, and an injective projection of $N_1$ onto some $(K/\CO)^d$, for $d=\dpr N$.

We claim that in fact $N_1\sub H$ as well: Indeed, by the definition of $\sigma$ we clearly have $\sigma(N_1)\subseteq N_1$, and hence $\sigma^{-1}(\sigma(N_1))\sub \sigma^{-1}(N_1)$. However, because $\ker \sigma\subseteq H$, we have
 $\sigma^{-1}(N_1)\subseteq H$, so $N_1\subseteq \sigma^{-1}\sigma(N_1)\sub H$. 
 
 Finally, $\ker \sigma$ is finite so  $\dpr (N_1)=\dpr (H)$.
  This ends the proof when all $H^i$ are finite. 

Assume that one of the $H^i$, say $H^1$, is infinite. Let $\pi_1:(K/\CO)^n\to K^{n-1}$ be the projection on the first coordinate. Let $H_1=\ker \pi_1\cap H$. 

\begin{claim}
	$\dpr H_1=\dpr H-1$.
\end{claim}
\begin{claimproof}
	
 Indeed, for every $y\in \pi_1(H)\subseteq K/\CO$, $\dpr(\pi_1^{-1}(y)\cap H)= \dpr H_1$, so  by sub-additivity,  $\dpr H_1\geq \dpr H-1$.
Also, we have $H\supseteq H_1\oplus H^1$ (since $H_1\subseteq \{0\}\times K^{n-1}$ and $H_1\sub K/\CO\times \{\bar 0_{n-1}\}$), so since $H^1$ is infinite then $\dpr(H_1)=\dpr H-1$.
\end{claimproof}

By identifying $H_1$ with a subgroup of $(K/\CO)^{n-1}$ we may find a definable $H_1'\subseteq H_1$, $\dpr H_1'=\dpr H_1$, and a  projection, call it $\tau^k:(K/\CO)^{n-1}\to (K/\CO)^k$,  onto some $k$ coodrinates among the last $n-1$ ones, where $k=\dpr H_1$, such that $\tau^k\restriction H_1'$ is injective.

The group $H':=H_1'\oplus H^1\sub H$ has the same rank as $H$ and the projection $\pi_1\times \tau^k:(K/\CO)^n\to (K/\CO)^{k+1}$ is injective on $H'$.
\end{proof}

\section{The infinitesimal subgroups $\nu_D$}\label{AS:nu}

In this section, we prove Lemma \ref{L: nu}. As usual, $\CK$ is our $(2^{\aleph_0})^+$-saturated $p$-adically closed field, $D$ a distinguished sort, and we fix a $D$-group $G$.   At the same time, and under the same assumptions we also show: 
\begin{lemma}\label{L: Dsets}
    Let $G$ be an infinite $D$-group defined in  $\CK$. Then there exists a $D$-critical set $X\sub G$ and a definable injection $f: X\to D^m$ for $m=\dpr(X)$. 
\end{lemma}

This lemma is needed for a technical reason that we explain briefly. In order to construct the infinitesimal type-definable group $\nu_D$, we require the existence of a $D$-critical set which is a $D$-set, a term introduced in \cite{HaHaPeGps}. A $D$-set set in an interpretable group $G$  is a $D$-critical set with a map $f: X\to D^m$ witnessing it satisfying the  additional requirement that the image, $f(X)$,  has \emph{minimal fibres} (a technical property that will not be introduced here). It follows immediately from the definition (Definition 4.11, \emph{loc. cit.}) that if $m=\dpr(X)$ then $X$ is a $D$-set. Thus, the lemma will allow us to apply results from \cite{HaHaPeGps}, without getting, any further, into the fine technicalities of $D$-sets. 

The proof of both Lemma \ref{L: nu} and Lemma \ref{L: Dsets} is split according to whether $D$ is $K$, $K/\CO$ or $\Gamma$. We remind the definition of the partial type $\nu_D$:\\

For  an $A$-generic $c\in X$ (i.e. $\dpr(c/A)=\dpr(X)$), let $\nu_X(c)$ be the partial type over $K$,  consisting of all $Y\subseteq X$, satisfying that  $c\in Y$, and $Y$ is definable over some set $B$ such that $\dpr(c/B)=\dpr(X)$. In \cite[Proposition 5.8]{HaHaPeGps}, we show that $\nu_D:=\nu_X(c)c^{-1}$ is a subgroup of $G$, independent of the choice of the $D$-set $X$, the function witnessing it $f$,  or the generic point $c$. 

\begin{remark}\label{R: nu lives on any definable subgroup witnessing}
Note that if $X\subseteq G$ is a $D$-set which is also a definable subgroup then $\nu_D\vdash X$. 
\end{remark}

\subsection{The case of $D=K$}
Assume that $D=K$.

 By \cite[Example 3.3]{HaHaPeVF} the valuation on $K$ makes it an SW-uniformity (the actual definition is immaterial for us here). As a consequence, by \cite[Proposition 4.6]{SimWal}, if $f:X\to K^m$ is a definable injection with $\dpr(X)=n$ then, at the possible cost of shrinking $X$ (but not its dp-rank), there also exists a definable injection $f:X\to K^n$ with $f(X)$ an open ball. So Lemma \ref{L: Dsets} holds for $K$-groups. We proceed to the proof Lemma \ref{L: nu} for $K$-groups:

\begin{fact}\label{F: nu in K}
Let $G$ be a definable $K$-group of $K$-rank $n>0$ and let $X\subseteq G$ be a definable subset and $f:X\to K^n$ a definable injection with $n$ equal the $K$-rank of $G$. For any $c\in X$, generic over $A$, 
\[\nu_X=\{f^{-1}(U)c^{-1}:U\subseteq K^n \text{ open  containing } f(c)\}.\]
\end{fact}
\begin{proof}
Assume everything is defined over a parameter set $A$. By \cite[Proposition 5.6]{HaHaPeGps}, for $c\in X$ $A$-generic $\nu_X(c)=f^{-1}(\mu(f(c))$, where $\mu(f(c))$ is the infinitesimal neighbourhood of $f(c)$ in the valuation topology on $K$. The result now follows. 
\end{proof}

\subsection{The case of $D=K/\CO$} \label{Ss: K/O}
Assume that  $D=K/\CO$.

By \cite[Theorem 7.11(3)]{HaHaPeGps}, there exists a definable subgroup $H\subseteq G$ with $\dpr(H)=n$ the $K/\CO$-rank of $G$, definably isomorphic to a subgroup of $((K/\CO)^r,+)$ for some $r>0$. By Lemma \ref{L:local homeo K/O}, we may assume, replacing $H$ with a subgroup of the same dp-rank,  that $r=n$. By \cite[Lemma 3.6]{HaHaPeGps}, we may further assume that $f(H)$ is a ball centered around $0$.

\begin{lemma}\label{F: nu in K/O}
Let $f:H\to (K/\CO)^n$ be a definable injective homomorphism over $A$, and $\dpr(H)=n$  the $K/\CO$-rank of $G$. Then 

 \[\nu_{K/\CO}=\{f^{-1}(U): U\subseteq (K/\CO)^n \text{ an open ball in $(K/\CO)^n$ centred at $0$}\}.\]
\end{lemma}
\begin{proof} 

Let 
\[
    \nu_1:=\{f^{-1}(U): U\subseteq (K/\CO)^n \text{ an open ball in $(K/\CO)^n$ centred at $0$}\}
\]
By definition,  $\nu_{K/\CO}=\nu_H(c)c^{-1}$ for some $A$-generic $c\in H$. Let $H_1 :=f(H)\le  (K/\CO)^n$.
Since $\dpr(H_1)=n$, by \cite[Lemma 3.6]{HaHaPeGps}, we may assume, shrinking $H$ if needed, that  $H_1$ is a ball in $(K/\CO)^n$. We claim  that $\nu_{K/\CO}=\nu_1$.

Let $U\subseteq H_1$ be an open ball, $0\in U$. By \cite[Proposition 3.8]{HaHaPeGps}, there exists a ball $Y\subseteq U+f(c)$, $f(c)\in Y$, definable over some $B\supseteq A$ such that $\dpr(f(c)/B)=n$. Since $H_1$ is a subgroup, we have $Y\subseteq H_1$. Now, as $f$ is group homomorphism, $f^{-1}(Y-f(c))=f^{-1}(Y)c^{-1}\subseteq U$, $c\in f^{-1}(Y)$,  and $\dpr(c/B)=n$. Thus, by the definition of $\nu_{K/\CO}$, we have $\nu_{K/\CO}\vdash f^{-1}(U)$, so $\nu_{K/\CO}\vdash \nu_1$. 

Similarly,  $\nu_1+c\vdash \nu_{H_1}(c)$, with the desired conclusion.   

\end{proof}

This completes the proof of Lemma \ref{L: nu} and of Lemma \ref{L: Dsets} for $K/\CO$-groups. 

\subsection{The case of $D=\Gamma$}
Assume that $D=\Gamma$. By cell decomposition, \cite{ClucPresburger} (see also \cite[Fact 2.4]{OnVi} for a more explicit formulation), and the fact that dimension in Presburger arithmetic coincides with dp-rank (\cite[Theorem 0.3]{Simdp}) we have:
\begin{fact}\label{F:minimal fibers in Gamma} 
Let $Y\subseteq \Gamma^m$ be a definable subset with $\dpr(Y)=n\leq m$ then there exists a definable $Z\subseteq Y$ with $\dpr(Z)=\dpr(Y)$ projecting injectively onto a definable subset of $\Gamma^n$ of full dp-rank.
\end{fact}

By using \cite[Corollary 3.7]{OnVi} and \cite[Lemma 4.2]{HaHaPeGps} repeatedly (as in the proofs of \cite[Proposition 5.6]{HaHaPeGps} and Lemma \ref{F: nu in K/O}) we get the following.
\begin{lemma}\label{L: inf neigh in Gamma}
    Let $G$ be a definable group and $g:Y\to \Gamma^n$ be a definable injection with $\dpr(Y)=n$, all $A$-definable. For any $A$-generic $c\in Y$, 
    \[\nu_Y(c)=\{g^{-1}(U): U\subseteq \Gamma^n \text{ a generalized box around $g(c)$}\}.\] 
\end{lemma}

As a result: 

\begin{lemma}\label{L: nu in Gamma}
Let $G$ be a definable $\Gamma$-group of $\Gamma$-rank $n>0$. There exists $X\subseteq G$, a $\Gamma$-critical set with $\nu_{\Gamma}\vdash X$, and $f:X\to \Gamma^n$ a definable injection satisfying:
\begin{list}{$\bullet$}{}
    \item $f(X)$ is a generalized box around $0$,
    \item $f(xy^{\pm 1})=f(x)\pm f(y)$ for any $x,y\in X$ with $xy^{\pm 1}\in X$ and
    \item $\nu_\Gamma=\{f^{-1}(U): U\subseteq \Gamma^n \text{ a generalized box around $0$}\}.$
\end{list}
\end{lemma}
\begin{proof}
    By \cite[Theorem 7.11(2)]{HaHaPeGps}, $\nu_\Gamma$ is definably isomorphic to a type-definable subgroup of $(\Gamma^r,+)$ for some $r>0$. By compactness, there is a definable subset $X\subseteq G$, $\nu_{\Gamma}\vdash X$, and a definable injection $f:X\to \Gamma^r$, such that $f(xy^{\pm 1})=f(x)\pm f(y)$ for all $x,y\in X$ with $xy^{\pm 1}\in X$.  By Fact \ref{F:minimal fibers in Gamma} we may further assume that $r=n$, the $\Gamma$-rank of $G$. 
    
    As $f_*\nu_\Gamma$ is a subgroup, by compactness there exists some $W_0\subseteq f(X)$, with $f_*\nu_\Gamma\vdash W_0$, such that $x-y\in f(X)$ for any $x,y\in W_0$. Assume everything is defined over some set $A$. Let $d\in W_0$ be an $A$-generic. By \cite[Lemma 3.4]{OnVi}, there exists a generalized box $B\subseteq W_0$ around $d$, so $B-d\subseteq f(X)$ is a generalized box around $0$. By shrinking $X$, but not its dp-rank, we may assume that $f(X)$ is a generalized box around $0$.
    
    As before, by compactness, there exists some definable subset $Y\subseteq X$, $\nu_{\Gamma}\vdash Y$, with $xy^{-1}\in X$ for any $x,y\in Y$. Assume that everything is still definable over $A$.
    
    Let $c\in Y$ be an $A$-generic. To show that \[\nu_\Gamma=\nu_X(c)c^{-1}=\{f^{-1}(U): U\subseteq \Gamma^n \text{ a generalized box around $0$}\},\] we will show they have the same realizations in some fixed sufficiently saturated elementary extension $\widehat \CK\succ \CK$. 
    
    Let $ac^{-1}\in \nu_\Gamma(\widehat \CK)$, with $a\in \nu_X(c)(\widehat \CK)$ and let $U\subseteq \Gamma^n$ be a generalized box around $0$.  By Lemma \ref{L: inf neigh in Gamma}, $a\in f^{-1}(U+f(c))$ so $f(a)\in U+f(c)$ and as $a,c\in Y$ we get that $ac^{-1}\in f^{-1}(U)$.
    
    The right-to-left inclusion is similar.
\end{proof}

This concludes the proof of Lemma \ref{L: nu} and of Lemma \ref{L: Dsets}.

\bibliographystyle{plain}
\bibliography{harvard}

\end{document}